\documentclass[a4paper, 11pt]{article}

\usepackage{soul}
\usepackage[utf8]{inputenc}
\usepackage{hyperref}
\usepackage{xcolor}
\usepackage{amsmath}
\usepackage{amssymb}
\usepackage{authblk}
\usepackage{lipsum}
\usepackage{geometry}
\usepackage{comment}
\newenvironment{proof}{\paragraph{Proof.}}{\hfill$\square$}
\newtheorem{theorem}{Theorem}[section]
\numberwithin{theorem}{section}
\newtheorem{lemma}[theorem]{Lemma}
\newtheorem{definition}[theorem]{Definition}
\newtheorem{corollary}[theorem]{Corollary}
\newtheorem{proposition}[theorem]{Proposition}
\newtheorem{remark}[theorem]{Remark}
\newtheorem{example}[theorem]{Example}

\newcommand{\Z}{{\mathbb{Z}}}
\newcommand{\BH}{\operatorname{BH}}

\begin{document}

\title{Generalized partially bent functions, 
generalized perfect arrays and 
cocyclic Butson matrices}
\date{}
\author[1]{J. A. Armario}
\author[2]{R. Egan}
\author[3]{D.~L.~Flannery}
\affil[1]{\small 
Departamento de Matem\'atica Aplicada I, 
Universidad de Sevilla,
Avda. Reina Mercedes s/n, 
41012 Sevilla, Spain}
\affil[2]{\small School of Mathematical Sciences, 
Dublin City University, Ireland}
\affil[3]{\small School of Mathematical and
Statistical Sciences,
 University of Galway,
Ireland}

\maketitle
\begin{abstract}
In a recent survey, Schmidt compiled 
equivalences between generalized 
bent functions, group invariant 
  Butson Hadamard matrices, and 
abelian splitting relative 
 difference sets.
We establish a broader network of 
equivalences
by considering Butson matrices 
that are cocyclic rather than strictly 
group invariant.
This result has several 
applications; for example,
to the construction of Boolean functions 
whose expansions are generalized partially 
bent functions, including cases where no 
bent function can exist.
\end{abstract}

\thispagestyle{empty}

\section{Introduction}
Let $f\colon \Z_{2}^{n} \rightarrow \Z_{2}$ 
be a Boolean function
with $n$ a positive integer, and set
$F(v) = (-1)^{f(v)}$ for
$v \in \Z_{2}^{n}$
(throughout, we view 
$\Z_t$ for an integer
$t>1$ as $\{0,1,\ldots, t-1\}$
under addition modulo $t$).
 The Walsh--Hadamard 
transform $\hat{F}$ of $F$ is defined by
\[
\hat{F}(u) = 
\sum_{v \in \Z_{2}^n}
 (-1)^{u\cdot v}
F(v),
\]
where $u\cdot v$ denotes the inner
product $uv^\top$ of $u$ and $v$.
The Boolean function $f$ is \emph{bent} 
if $|\hat{F}(u)|$ is constant for all
$u \in \Z_{2}^{n}$. Parseval's theorem
(see, e.g., \cite[(8.36), p.~322]{ParsevalThm})
gives 
\[
\sum_{v \in \Z_{2}^{n}}\hat{F}(v)^2 = 2^{2n}.
\]
Hence, $f$ can be bent only if $n$
is even.

A bent function is so-called because it is as 
far from being linear as possible. 
These highly non-linear functions are extremely 
useful for encryption, 
offering a robust defence against linear 
cryptanalysis~\cite[Chapter 3]{Horadam}.

Bent functions are equivalent to certain
Hadamard matrices and difference sets;
see, e.g., 
\cite[Lemma~14.3.2]{MW-S} and
\cite[Corollary~3.30]{Horadam}. 
The concept has been generalized, yielding
 equivalences between various associated
objects. Indeed, our paper is inspired by 
Schmidt's survey~\cite{Sch19}, which describes 
equivalences between generalized bent functions, 
group invariant Butson Hadamard matrices, 
and  splitting relative difference sets. 
There is also a connection to perfect arrays, 
not covered in \cite{Sch19}.

Our goal is to study how the aforementioned 
equivalences are affected when the property of 
being group invariant is broadened to 
\emph{cocyclic development}.
We now outline the paper. 
Preliminary definitions and results 
are given in Section~\ref{sectionB}. 
Section~\ref{sec:further} is devoted to
generalized perfect 
arrays and generalized partially bent 
functions. In Section~\ref{sec:main}, 
we prove the main theorem: a
series of equivalences between 
 cocyclic Butson 
 Hadamard matrices, generalized perfect 
 arrays, non-splitting relative 
 difference sets, generalized 
 plateaued functions, and generalized 
 partially bent functions. 
(For certain parameters, the 
 equivalences that we exhibit have 
 those in \cite{Sch19} as special cases.)
Section~\ref{sec:examples} 
 contains examples
 illustrating the main
 theorem.

\section{Background}
\label{sectionB}

We adopt the following definition 
 from \cite{Sch19}.
For positive  integers $q, m, h$, and
$\zeta_k$ the complex $k^{\rm th}$ 
root of unity $\exp{(2\pi \sqrt{-1} /k)}$,
 a map $f\colon {\mathbb{Z}}_{q}^m 
\rightarrow {\mathbb{Z}}_{h}$ is 
a {\em generalized bent function (GBF)} if
\[     
\Big|\sum_{x\in{\mathbb{Z}}_{q}^m} 
\zeta_h^{f(x)} \zeta_q^{-w\cdot x}
\Big|^2=q^m\quad \forall \, w\in 
{\mathbb{Z}}_{q}^m,
\]
where $|z|$ as usual denotes the 
modulus of $z\in \mathbb C$. 
Thus, a GBF for $q=h=2$ and even $m$ 
 is a bent function. For $h=q$, 
Kumar, Scholtz, and Welch~\cite{KSW85} 
prove that GBFs exist if $m$ is 
even or $q\not \equiv 2 \, \bmod 4$.
However, no GBF with $h=q$, $m$ odd,
 and $q \equiv 2 \, \bmod 4$ is 
known~\cite[p.~2]{LeungSchmidt19}.

A further generalization
is relevant to our paper. 
If the values of
\[     
\Big|\sum_{x\in{\mathbb{Z}}_{q}^m} 
\zeta_h^{f(x)} \zeta_q^{-w\cdot x}
\Big|^2
\]
as $w$ ranges over $\Z_{q}^{m}$
lie in $\{ 0,\alpha\}$ 
for a single non-zero $\alpha$,
then $f$ is a
\emph{generalized plateaued function}. 
Mesnager, Tang, and Qi~\cite{MTQ17}
discuss such functions under the 
conditions that $q$ is prime and 
$h$ is a $q$-power.
They call $f$ an 
\emph{$s$-generalized plateaued function} 
when $\alpha$ has the form 
$q^{m+s}$.

We examine the role of GBFs and generalized 
plateaued functions within 
cocyclic design theory~\cite{ADT,Horadam}.
Some requisite definitions follow.
Let $G$ and $U$ be finite groups, with $U$ 
abelian. A map 
$\psi\colon G\times G\rightarrow U$ such that
\[
\psi(a,b)\psi(ab,c)=\psi(a,bc)\psi(b,c) 
\quad\forall \, a,b,c\in G
\]
is a \emph{cocycle} (\emph{over $G$},
\emph{with coefficients in $U$}). 
All cocycles $\psi$ 
are assumed to be {\em normalized},
meaning that $\psi(1,1)=1$. 
For any (normalized) map 
$\phi\colon G\rightarrow U$, 
the cocycle $\partial\phi$ defined by 
$\partial\phi(a,b)=
\phi(a)^{-1}\phi(b)^{-1}\phi(ab)$ is 
a {\em coboundary}. The set of
all cocycles 
$\psi\colon\allowbreak G\times G
\rightarrow U$ equipped with pointwise 
multiplication is an abelian group, 
$Z^2(G,U)$. 
Factoring out $Z^2(G,U)$ by 
the subgroup $B^2(G,U)$ of coboundaries 
gives the \emph{second cohomology group},
$H^2(G,U)$. The elements of $H^2(G,U)$,
namely cosets of $B^2(G,U)$, are 
\emph{cohomology classes}.
Each $\psi\in Z^2(G,U)$ 
is displayed as a {\em cocyclic matrix}
$M_\psi$. That is, under an 
indexing of rows and columns by 
the elements of $G$, the $|G|\times |G|$ 
matrix $M_\psi$ has entry $\psi(a,b)$ 
in position $(a,b)$.
We focus on abelian
$G$ and cyclic $U$;
say $G={\mathbb{Z}}_{s_1}
\times\cdots\times {\mathbb{Z}}_{s_m}$  
and $U=\langle \zeta_h\rangle\cong \Z_h$, 
where $\langle \zeta_{h}\rangle:=
\{\zeta_{h}^{i} \; | \; 0 \leq i \leq h-1\}$
is generated (multiplicatively)
 by $\zeta_h$.

Denote the set of $n \times n$ matrices with 
entries in a set $S$ by $\mathcal{M}_{n}(S)$.
A matrix $M\in  \mathcal{M}_{n}
(\langle \zeta_{k} \rangle)$ is a 
{\em Butson 
(Hadamard) matrix} if
$M M^*=nI_n$, where $I_n$ is 
the $n\times n$ identity matrix and 
$M^*$ is the complex conjugate 
transpose of $M$. We write $\BH(n,k)$ to
denote the (possibly empty) 
set of all Butson matrices
in $\mathcal{M}_{n}
(\langle \zeta_{k} \rangle)$.
For example, at every order $n$ we 
have the Fourier matrix 
$\big[
\zeta_n^{(i-1)(j-1)}
\big]_{i,j=1}^n\in \BH(n,n)$.
Hadamard matrices of order $n$
are the elements of $\BH(n,2)$.
We quote a number-theoretic 
constraint on the existence of
elements of $\BH(n,k)$.

\vspace{3.5pt}

\begin{theorem}[{\cite[Theorem~2.8.4]{ADT}}]
If $\BH(n,k)\neq \emptyset$ and 
$p_{1},\ldots,p_{r}$ are the primes 
dividing $k$, then 
$n = a_{1}p_{1} + \cdots + a_{r}p_{r}$ 
for some non-negative integers 
 $a_1,\ldots , a_r$.
\end{theorem}

\vspace{3.5pt}

Matrices $H, H'\in \mathcal{M}_{n}
(\langle \zeta_{k} \rangle)$
are \textit{equivalent} if 
$PHQ^{*} = H'$ for monomial matrices 
 $P,\allowbreak Q\in \mathcal{M}_n
( \langle \zeta_k\rangle\cup \{0\})$.
This equivalence relation 
induces a partition of $\BH(n,k)$.

Our interest is in 
cocyclic Butson matrices.
Let $G$ be a group of order $n$.
A cocycle
$\psi\in Z^2(G,\langle \zeta_k\rangle)$ 
such that $M_\psi\in \BH(n,k)$ is 
{\em orthogonal}.
In particular, group invariant 
Butson matrices are cocyclic.
The orthogonal cocycles involved 
here are coboundaries, as we now explain.
A matrix $X\in \mathcal{M}_{n}(U)$ 
is \emph{group invariant}, over 
$G$, if $X=[x_{a,b}]_{a,b\in G}$ and 
$x_{ac,bc} = x_{a,b}$ for all $a,b,c\in G$.
Such a matrix $X$ 
is equivalent to a
\emph{group-developed matrix} 
$[\chi(ab)]_{a,b\in G}$
for some map
$\chi\colon G\rightarrow U$;
 in turn $[\chi(ab)]$ is equivalent to
$M_{\partial\chi}$. 
A group-developed Butson matrix 
has constant row and column sum
(in $\mathbb C$).
There are much stronger restrictions 
on the existence of group-developed 
elements of $\BH(n,k)$.

\vspace{3.5pt}

\begin{lemma}[{\cite[Lemma~5.2]{EFO15}}]
\label{GrpDevConstraint}
Set $r_{j} = \mathrm{Re}(\zeta_k^j)$ 
and $s_{j} =\mathrm{Im}(\zeta_k^j)$. A
matrix in 
$\BH(n,k)$ with constant row and 
column sums exists only if there
are $x_{0},\ldots, x_{k-1} \in 
\{0,1,\ldots,n\}$ such that
$\big(
\sum_{j=0}^{k-1} r_{j}x_{j}\big)^{2} +
\big(\sum_{j=0}^{k-1}s_{j}x_{j}\big)^{2} = n$
and $\sum_{j=0}^{k-1} x_{j} = n$.
\end{lemma}

\vspace{3.5pt}

It follows from 
Lemma~\ref{GrpDevConstraint} that if 
$k=2$ then $n$ is an integer square, 
and if $k=4$ then $n$ is the sum of 
two integer squares.

Cocyclic designs give rise to relative
difference sets, and 
vice  versa~\cite[Sections~10.4, 15.4]{ADT}.
 Let $E$ be a group with normal 
 subgroup $N$, where $|N|=n$ and 
  $|E:N|=v$.  
 A \emph{$(v,n,k,\lambda)$-relative 
 difference set in $E$ relative to $N$} 
 (the \textit{forbidden subgroup})
is a $k$-subset $R$ of 
a transversal for $N$ in $E$ such that
$|R\cap xR|=\lambda$
for all $x\in E\setminus N$.
We call $R$  {\em abelian} if $E$ is 
abelian, and {\em splitting}
if $N$ is a direct factor of $E$.

The final piece of 
 background concerns arrays.
Let ${\bf s}= (s_1,\ldots,s_m)$ be an 
$m$-tuple of integers $s_i>1$,
and let $G={\mathbb{Z}}_{s_1}
\times\cdots\times {\mathbb{Z}}_{s_m}$. 
A {\em $h$-ary ${\bf s}$-array}
is just a set map 
$\phi \colon 
G\rightarrow \allowbreak \Z_h$
(normalized when necessary). 
If $h=2$, then the array is 
\textit{binary}.
For $w\in G$,  the 
{\em periodic autocorrelation 
of $\phi$ at shift $w$},
denoted $AC_{\phi}(w)$, is 
defined by
\[
AC_{\phi}(w)=
\sum_{g\in  G} 
\zeta_h^{\phi(g)-\phi(g+w)}.
\]
If $AC_{\phi}(w) = 0$ for all $w \neq 0$, 
then $\phi$ is \textit{perfect}.

\vspace{3.5pt}

\begin{lemma}\label{corredftWT}
Let $D_m$ be the $m^{th}$ Kronecker 
power of the $q\times q$ Fourier matrix, 
i.e.,  $(D_m)_{i,j}=
\zeta_q^{\alpha_{i-1} \cdot \alpha_{j-1}}$,
where $\alpha_0=(0,\ldots,0), 
\alpha_1=(0,0,\ldots, 1),\ldots, 
\alpha_{q^m-1}=(q-1,\ldots,q-1)$. 
Then, for any map
$\phi\colon  {\mathbb{Z}}_{q}^m \rightarrow 
{\mathbb{Z}}_{h}$,
\[
(AC_{\phi}(\alpha_0),\ldots, 
AC_{\phi}(\alpha_{q^m-1})) D_m =
\big(
\Big|\! \sum_{x\in{\mathbb{Z}}_{q}^m} 
\zeta_h^{\phi(x)} \zeta_q^{-\alpha_0 \cdot x}
\Big|^2, \ldots, \Big|\! 
\sum_{x\in{\mathbb{Z}}_{q}^m} 
\zeta_h^{\phi(x)} 
\zeta_q^{-\alpha_{q^m-1} \cdot x}
\Big|^2\big) .
\]
\end{lemma}
\begin{proof}
We adapt the proof of the 
lemma (for Boolean functions)
in \cite[Section~2]{Car93}. 
First, 
\[
\sum_{i\geq 0} AC_\phi(\alpha_i) 
\zeta_q^{\alpha_i\cdot \alpha_j}= \sum_{i\geq 0}  
\sum_{k\geq 0}\zeta_h^{\phi(\alpha_k)-
\phi(\alpha_k+\alpha_i) }  
\zeta_q^{\alpha_i\cdot \alpha_j}.
\]
After replacing $\alpha_i$ by 
$\alpha_i-\alpha_k$, the double 
summation becomes
\begin{align*}
\sum_{i\geq 0}  \sum_{k\geq 0}
\zeta_h^{\phi(\alpha_k)-\phi(\alpha_i) }  
\zeta_q^{\alpha_i\cdot \alpha_j- 
\alpha_k\cdot \alpha_j} 
& =\sum_{k\geq 0}\zeta_h^{\phi(\alpha_k) }  
\zeta_q^{- \alpha_k\cdot \alpha_j} 
\sum_{i\geq 0}\zeta_h^{-\phi(\alpha_i) }  
\zeta_q^{\alpha_i\cdot \alpha_j}\\
& =  \Big|\sum_{x\in{\mathbb{Z}}_{q}^m} 
\zeta_h^{\phi(x)} 
\zeta_q^{-\alpha_{j} \cdot x} \Big|^2,
\end{align*} 
as required.
\end{proof}

\vspace{10pt}

Our fundamental motivating result is
extracted mostly from \cite{Sch19}.

\vspace{2pt}

\begin{theorem} \label{prop-ben-perfe-rds}
Let $f\colon {\mathbb{Z}}_{q}^m 
 \rightarrow {\mathbb{Z}}_{h}$ be a map.
 The following  are equivalent$:$
\begin{enumerate}
\item[{\rm (1)}] $f$ is a GBF$;$
\item[{\rm (2)}] 
$M_{\partial f}\in \mathrm{BH}(q^m,h);$ 
\item[{\rm (3)}] $f$ is a perfect 
$h$-ary $(q,\ldots, q)$-array. 
\end{enumerate}
Additionally, if $h$ is  
prime and divides $q^m$, then 
 {\rm (1)}--{\rm (3)} 
are equivalent to
\begin{enumerate}
\item[{\rm (4)}] $\{(f(x),x) 
\; | \; x\in {\mathbb{Z}}_{q}^m\}$ 
 is a splitting 
$(q^m,h,q^m,q^m/h)$-relative difference set in 
${\mathbb{Z}}_{h}\times {\mathbb{Z}}_{q}^m $.
\end{enumerate}    
\end{theorem}
\begin{proof}
The equivalences $(1)\Leftrightarrow (2) 
\Leftrightarrow (4)$ come from 
Propositions~2.3  and 2.7 of \cite{Sch19}
($h$ prime is a sufficient
condition to ensure $(2)\Rightarrow (4)$). 
Lemma \ref{corredftWT}
implies $(1)\Leftrightarrow (3)$.
\end{proof}

\vspace{10pt}

We investigate the effect on  
Theorem~\ref{prop-ben-perfe-rds} when 
non-coboundary cocyclic Butson matrices,
generalized perfect arrays, and
non-splitting abelian 
relative difference sets 
 are considered in
 (2), (3), (4), respectively.
To this end, we need some material
of a  more specialized nature,
which is presented 
over the next two sections.

\section{More on arrays and bent 
functions}\label{sec:further}

There is an
equivalence between binary arrays 
and non-splitting abelian relative 
difference sets, as set out in
\cite{Jed92}.
Subsequently, a bridge 
to the theory of cocyclic Hadamard 
matrices was identified~\cite{Hug00}. 
The main tool here is the
notion of a generalized perfect 
binary array (GPBA).
Guided by \cite[Section~3]{GOBA}, 
 we extend the notion of GPBA from 
 binary to $h$-ary arrays, $h\geq 2$, 
and show how this conforms with a variant
of  bent functions.

\vspace{2pt}

\begin{definition}\label{defexph-ary}
{\em 
Let $\phi\colon G \rightarrow \mathbb{Z}_{h}$
be an {\bf s}-array, 
where ${\bf s}= (s_1,\ldots,s_m)$ and 
$G={\mathbb{Z}}_{s_1}
\times\cdots\allowbreak \times {\mathbb{Z}}_{s_m}$.
Let ${\bf z}=(z_1,\ldots,z_m)\in 
\{ 0,1\}^m$. 
The {\em expansion of $\phi$
 of type ${\bf z}$} is the map $\phi'$ 
from $E:=\allowbreak \mathbb{Z}_{(z_1(h-1)+1)s_1}
\times \cdots\times
\mathbb{Z}_{(z_m(h-1)+1)s_m}$ to
$\mathbb{Z}_{h}$ defined by
\[
\phi'\colon (g_1, \ldots , g_m) \mapsto
\phi(a) +b  \, \bmod h ,
\]
where $b=\sum_{i=1}^m \lfloor g_i/s_i\rfloor$
and $a\equiv (g_1, \ldots , g_m) \, 
 \bmod  {\bf{s}}$, 
i.e., $a =  (g_1 \, \bmod s_1, 
\ldots, \allowbreak 
g_m  \, \bmod s_m)$.}
\end{definition}

We distinguish two subgroups of the extension 
group $E$ in Definition~\ref{defexph-ary}:
\[\begin{array}{l}
L=\{(g_1,\ldots,g_m)\in E \hspace{.5pt} \ | \ 
g_i=y_is_i \
\mbox{with $0\leq y_i< h$ if $z_i=1$, 
and $y_i=0$ if $z_i=0$} \},\\
K=\{ (g_1,\ldots,g_m)\in L\ |\
\mbox{$\sum_i (g_i/s_i)
\equiv 0  \, \bmod h$}\}.
\end{array}
\]
Note that
\begin{itemize}
\item $L\cong\Z_h^n$ where
$n=\mathrm{wt}({\bf z})=\sum_i z_i$;
\item $E/L\cong G$;
\item if ${\bf z}\neq {\bf 0}$ then 
$L/K= \langle 
(0,\ldots , 0, s_i, 0 , \ldots , 0) +K\rangle
\cong\Z_h$, for any $i$
such that $z_i=1$.
\end{itemize}

\begin{lemma}\label{lem-phi}
Let $\phi$ be  a $h$-ary 
$(s_1,\ldots , s_m)$-array with expansion 
$\phi'\colon E\rightarrow \Z_h$. 
If $e\in E$ and $g=\allowbreak
(g_1,\ldots , g_m)\in L$, 
then $\phi'(e+g)\equiv \phi'(e)+b \, \bmod h$ 
where $b=\sum_i g_i/s_i$. 
\end{lemma}
\begin{proof}
This is routine, from the definitions.
\end{proof}

\vspace{2.5pt}

\begin{corollary}
\label{lemmaach}
$AC_{\phi'}(g)=
\zeta_h^{-b} |E|$ for any $g\in L$.
\end{corollary}

\begin{definition} \label{dfn-GPhA}
{\em 
A \mbox{$h$-ary} $\bf s$-array $\phi$ 
with expansion $\phi'\colon E\rightarrow \Z_h$
of type ${\bf z}$ is {\em generalized perfect}
if  $AC_{\phi'}(g)=\allowbreak 0$ for all 
$g \in E\setminus L$;
in short, $\phi$ is a 
$\mathrm{GPhA}({\bf s})$ of type ${\bf z}$.
 We write $\mathrm{GPhA}(c^m)$ when
$\mathbf{s}=(c,\ldots, c)$
of length $m$ for a constant $c$.}
\end{definition}

\vspace{2.5pt}

So a GPhA$({\bf s})$ of type
${\bf 0}$ is exactly
a perfect $h$-ary  ${\bf s}$-array.

\vspace{2.5pt}

\begin{definition}[cf.{~\cite[Definition~2.2]{WZ07}}]
\label{dfgpbf}
{\em 
A map $f\colon {\mathbb{Z}}_{q}^m
\rightarrow \allowbreak 
{\mathbb{Z}}_{h}$
such that 
$|AC_f(x)|\in\{ 0, q^m\}$
for all $x\in \mathbb{Z}_{q}^m$
is a {\em generalized partially bent 
function} (GPBF).}
\end{definition}

\vspace{2.5pt}

Let $\phi$ be a $h$-ary $(q,\ldots, q)$-array.
By Corollary~\ref{lemmaach}, 
if $\phi$ is generalized perfect, then 
$\phi'$ is generalized partially bent.
However, the converse does not hold, as
evidenced by the following simple
example.
Define $\phi\colon  {\mathbb{Z}}_{2}^2 
 \rightarrow {\mathbb{Z}}_{2}$ 
 by $\phi(0,1)=1$
 and  $\phi(0,0)=\phi(1,0)=\phi(1,1)=0$. 
 The expansion of $\phi$ of type $\bf 1$
 is a GPBF, but $\phi$ is not a 
 $\mathrm{GP2A}(2^2)$ of type $\bf 1$
 (writing $\bf 1$ for the all $1$s vector).
 We obtain the converse by
 imposing more conditions.
 
 \vspace{2.5pt}

\begin{proposition} \label{pro-gpbf-gpha}
Let $\phi$ be an array 
 ${\mathbb{Z}}_{h}^m  \rightarrow 
{\mathbb{Z}}_{h}$
 such that for each
 $y=(y_1,\ldots , y_m)\in
 \allowbreak {\mathbb{Z}}_{h}^m
 \setminus \{\mathbf{0}\}$ 
 with $\sum_i y_i\equiv 0 \, \bmod h$, 
 there exists 
$x=(x_1,\ldots,x_m)\in {\mathbb{Z}}_{h}^m$ 
satisfying
 \begin{equation}
 \label{eq-phi-lineal}
\phi(x+y)+
\sum_i \left\lfloor (x_i+y_i)/h\right\rfloor 
\not \equiv \phi(x)+\phi(y) \, \bmod h. 
 \end{equation}
Then the expansion $\phi'$ of  $\phi$ 
 of type  $\mathbf{1}$ is a GPBF 
 if and only if $\phi$ is a $GPhA(h^m)$
 of type $\mathbf{1}$.
\end{proposition}
\begin{proof}
In this proposition, $E = \Z_{h^2}^m$ and 
$L = \{ 0, h, \ldots , (h-1)h\}^m
\cong \Z_h^m$. 
Suppose that $\phi'$ is a GPBF. 
Then $\phi$ is a $\mathrm{GPhA}(h^m)$
if $|AC_{\phi'}(g)|<h^{2m}$ 
for all  $g\in E\setminus L$.
So we prove that
$\phi'(w)-\phi'(w+g)
\not \equiv \phi'(x)-\allowbreak 
\phi'(x+g)\, \bmod h$
for some $w, x\in E$.
Taking $w=0$, and assuming that 
$\phi$ is normalized, this non-congruence
becomes
$\phi'(x+g)\not \equiv 
 \phi'(x)+\allowbreak \phi'(g)$.

Suppose that  
\[
\phi'(0)-\phi'(g), \, 
  \phi'(g)-\phi'(2g) , \, 
 \ldots , \, 
 \phi'((h-1)g) -\phi'(h g)     
\]
are all congruent modulo $h$  
 (otherwise, the required $x$
 may be found as a multiple of $g$).
Adding these $h$ terms gives
\[
\phi'(0)-\phi'(hg)
\equiv 0 \, \bmod h 
\ \Rightarrow \ \phi'(hg) 
\equiv 0 \, \bmod h.
\]
Consequently
$\sum_i g_i\equiv  0\, \bmod h$.

If $g=(g_1,\ldots, g_m)$ with
$0\leq g_i< h$, then the 
right-hand side of
  \eqref{eq-phi-lineal} for $y=g$
is $\phi'(x)+\allowbreak \phi'(g)$, 
and the left-hand side 
is $\phi'(x+g)$; so we are done.

Now let $g=a+l$ with  
$a = (g_1 \, \bmod h, \ldots,  
g_m  \, \bmod h)$ and $l\in L$. 
Then $\sum_i a_i\equiv  0$,
because $hg = ha$ in $E$.
Using Lemma~\ref{lem-phi},
and adding $b=\sum_i l_i/h$ 
to both sides of 
\eqref{eq-phi-lineal} for $y=a$,
we see that 
$\phi'(x+g)\not \equiv 
\phi'(x) + \phi'(g)$. 
This completes the proof.
\end{proof}

\section{Equivalences between arrays, 
bent functions, and 
associated combinatorial objects}
\label{sec:main}

Let $\bf s$, $\bf z$, $G$, $K$, $L$, $E$ 
be as in Definition \ref{defexph-ary} 
and its environs, with 
${\bf z}\neq {\bf 0}$. 
We have a short exact sequence
\begin{equation}
\label{sejed}
1 \longrightarrow \langle \zeta_h\rangle 
\stackrel{\iota}{\longrightarrow}
E/{K} \stackrel{\beta}{\longrightarrow} 
G \longrightarrow 0 ,
\end{equation}
where $\beta(g+K)\equiv g \, \bmod  {\bf{s}}$
and $\iota$ sends $\zeta_h$ to a
generator of $L/K\cong \mathbb{Z}_h$.
In the standard way we extract a
cocycle $\mu_{\bf z}\in 
Z^2(G,\langle \zeta_h\rangle)$,
from \eqref{sejed}, 
depending on the choice of a transversal 
map $\tau \colon\allowbreak  
G\rightarrow \allowbreak E/K$.  
Set $\tau(x) = x+K$ (a mild
abuse of notation), so  that
$\beta\circ \tau = \mathrm{id}_G$; 
then $\mu_{\bf z}(x, y)= 
\iota^{-1}(\tau(x)+ \tau(y) - 
\allowbreak \tau(x+y))$.

\vspace{2.5pt}

\begin{proposition}[cf.~{\cite[Lemma~3.1]{Hug00}}]
\label{pro-fz}
Define $\gamma_t\in 
Z^2(\mathbb{Z}_t,\langle\zeta_h\rangle)$
by $\gamma_t(j,k)= 
\zeta_h^{\lfloor (j+k)/t\rfloor}$.
Then
\begin{itemize}
\item[{\em (i)}]
$\mu_{\bf z}(x,y)=\prod_{i \,
{\tiny \mbox{with}} \, z_i=1}
\gamma_{s_i}(x_i,y_i);$
\item[{\em (ii)}]
$\mu_{\bf z}\in 
B^2(G,\langle \zeta_h\rangle)$ if and only
if $s_{i}$ is coprime to $h$
whenever $z_{i} =1$.
\end{itemize}
\end{proposition}

\vspace{2.5pt}

In the opposite direction,
each cocycle $\psi\in 
Z^2(G, \langle \zeta_h\rangle)$  
determines a central extension 
$E_{\psi}$ of $\langle \zeta_h\rangle$
by $G$: namely, the group with elements
$\{ (\zeta_h^j,g) \mid 
0\leq j< h,\,g\in G\}$
and multiplication defined by 
$(u,g)(v,h)=(uv\hspace{.5pt} \psi(g,h),gh)$.
More properly, the central extension is
the short exact sequence
\begin{equation}\label{StdSES}
1\longrightarrow \langle \zeta_h\rangle 
\stackrel{\iota'}{\longrightarrow} E_{\psi} 
\stackrel{\beta'}{\longrightarrow} G 
\longrightarrow 0,
\end{equation}
where $\iota'(u)=(u,0)$ and 
$\beta'(u,x)=x$.

The next two results mimic
Proposition~4 and
Lemma~3 of \cite{GOBA}, 
respectively.

\vspace{3.5pt}

\begin{proposition}
\label{GammaRef}
If $\mu_{\bf z}$ and 
$\psi\in Z^2(G,\langle \zeta_h\rangle)$
are in the same cohomology class, say 
$\psi=\mu_{\bf z}\partial\phi$, then 
{\em (\ref{sejed})} and 
 {\em (\ref{StdSES})} 
are equivalent as short exact
sequences. Specifically, for the 
transversal map $\tau$
as defined before
Proposition{\em ~\ref{pro-fz}},
the map $\Gamma$ sending
$(u,x)\in E_\psi$ to  
$\iota(u\phi(x)^{-1})+\tau(x)
\in E/K$ is an isomorphism 
that makes the diagram
\[
\begin{array}{ccccccccc}
1 & \longrightarrow & \langle \zeta_h\rangle 
& \stackrel{\iota'}{\longrightarrow} & 
E_{\psi} & \stackrel{\beta'}{\longrightarrow} 
& G & \longrightarrow & 0\\[2mm]
& & \| & & 
{\scriptsize \mbox{$\Gamma$}} 
\big\downarrow\phantom{\Gamma} & & \|
 & &  \\[2mm]
1 & \longrightarrow & 
\langle \zeta_h\rangle 
& \stackrel{\iota}{\longrightarrow} & E/K 
& \stackrel{\beta}{\longrightarrow} & G & 
\longrightarrow & 0
\end{array}
\]
commute.
\end{proposition}
\begin{remark}
\label{CYAMultorAdd}
{\em 
In Proposition~\ref{GammaRef}, the 
 $\phi$ has multiplicative
target group $\langle \zeta_h\rangle$.
When considering $\phi$ as
an array, we may replace the multiplicative
group $\langle \zeta_h\rangle$
by the additive group $\Z_h$, without
bothering to change notation.
Likewise, note that $E_\psi$ is treated 
multiplicatively,
whereas $E$ and its subgroups
and quotients are treated additively.}
\end{remark}
\begin{lemma}
\label{GammaTransmit}
Assuming the set-up of Proposition{\em ~\ref{GammaRef}}, $\Gamma$ maps 
$\{(1,x)\ | \ x\in G\}\subseteq E_{\psi}$ 
onto $\{g+K\in E/K \mid \phi'(g)
\equiv 0 \, \bmod h \}$.
\end{lemma}
\begin{proof}
As $\phi'$ is constant on each coset of 
$K$ in $E$ by Lemma~\ref{lem-phi}, the 
stated subset of $E/K$ is well-defined.
If $\phi(x) = \zeta_h^j$
then $\Gamma((1,x))= -jy + x+K$ where
$\iota(\zeta_h) = y+K$ generates
$L/K$. Remember that $y$ may be 
chosen as $(0,\ldots,0,s_{i},0,\ldots,0)$ 
for some $i$.
Again by Lemma~\ref{lem-phi},
$\phi'(-j y + x)= j -
j \sum_i(y_i/s_i)
\equiv 0  \, \bmod h$. 
Conversely, suppose that 
$\phi'(g)=0$. 
Put $a\equiv g \, \allowbreak 
\bmod {\bf s}$ and 
 $b\equiv \sum_i\lfloor g_i/s_i\rfloor
\, \bmod h$;
so $\phi(a)=\phi'(g)-b\equiv -b
\allowbreak \, \bmod h$. 
Therefore, because
$g-\allowbreak 
a-(0,\ldots, 0, bs_i , 0, \ldots, 0)
\in K$, we get that
$g+K= \iota(\phi(a)^{-1})+a+K
=\allowbreak \Gamma((1,a))$.
\end{proof} 
\begin{remark}
{\em $\{(1,x)\ | \ x\in G\}$ is a 
full transversal for the cosets of 
$\langle \zeta_h\rangle$ in
$E_{\psi}$.}
\end{remark}

\vspace{7pt}

Next we present a couple of lemmas 
about special subsets of $E$, to be used 
in the proof of the impending theorem.
For $0\leq i\leq h-1$, 
define $N_{\phi'}^i=\{g\in E \ |\  
\phi'(g)\equiv i \, \bmod h\}$ and 
$L_i=\{g\in L \ | \allowbreak 
\sum_k (g_k/s_k) 
\equiv i \, \bmod h\}$.  
\begin{lemma}\label{lemmaireduce0}
$N_{\phi'}^{i}+L_{j}=
N_{\phi'}^{i+j}$
(elementwise sum in $E$),
reading indices modulo $h$.
\end{lemma}
\begin{proof}
If $x\in N_{\phi'}^i$ and $g\in L_j$,
  then $\phi'(x+g)\equiv
  \phi'(x) +\sum_k (g_k/s_k)
  \equiv i+j$
  by Lemma~\ref{lem-phi}. Hence
 $N_{\phi'}^{i}+L_j\subseteq 
 N_{\phi'}^{i+j}$.
 Since $-L_j = L_{h-j}$,
 this containment implies that
 $N_{\phi'}^{i+j}-L_j\subseteq 
 N_{\phi'}^{i}$, and so
 $N_{\phi'}^{i+j} = 
  N_{\phi'}^{i}+L_j$.
 \end{proof}
 
\begin{lemma}\label{ndslemma}
For all $i,j$ and $e\in E$,
$|N_{\phi'}^i\cap(e+N_{\phi'}^{i})|=
|N_{\phi'}^j\cap (e+N_{\phi'}^{j})|$.
\end{lemma}
\begin{proof}
The equation $x-y=\allowbreak e$ 
has precisely 
$|N_{\phi'}^i\cap(e+N_{\phi'}^{i})|$
solutions
$(x,y)\in N_{\phi'}^i\times N_{\phi'}^i$.
By Lemma~\ref{lemmaireduce0}, 
for $g\in L_{j-i}$ each such 
$(x,y)$ gives a solution 
$(\tilde{x}, \tilde{y}) = 
(x+g,y+g)\in 
N_{\phi'}^j\times N_{\phi'}^j$ of the 
equation $\tilde{x}-\tilde{y}=\allowbreak 
e$. 
Thus $|N_{\phi'}^i\cap(e+N_{\phi'}^{i})|
\leq |N_{\phi'}^j\cap(e+N_{\phi'}^{j})|$.
Swapping $i$ and $j$ gives the equality.
\end{proof}

\vspace{12pt}

We also need a fact about
vanishing
sums of roots of unity (see,
e.g., \cite[Lemma 2.8.5]{ADT}).

\vspace{3pt}

\begin{lemma}\label{lemmahprimevanish}
For prime $h$, if 
$\sum_{i=0}^{h-1} \alpha_i\zeta_h^i =0$ with 
$\alpha_i\in \Z$, then 
$\alpha_0=\alpha_1=\cdots=\alpha_{h-1}$.
\end{lemma}

\begin{theorem}
\label{th-gpha-rds}
Let $\phi$ be a $h$-ary ${\bf s}$-array of 
type ${\bf z}\neq {\bf 0}$, where 
$h$ is a prime dividing $v:=|G|=\prod_i s_i$
(Definition{\em ~\ref{defexph-ary}}), 
and let 
\[
R=\{g+K\in E/K \mid  \phi'(g)\equiv 0
\, \bmod h\}.
\]
Then $\phi$ is a $GPhA({\bf s})$ of type 
${\bf z}$ if and only if $R$ is a 
$(v,h,v,v/h)$-relative difference 
set in $E/K$ 
with forbidden subgroup $L/K$.
\end{theorem}
\begin{proof}
For $e\in E$ and $0\leq k<h$, 
define $B_k^{(e)}= \sum_{i=0}^{h-1}
|N_{\phi'}^i\cap(N_{\phi'}^{i-k} - e)|$.
We readily see that
\[
AC_{\phi'}(e)=
\sum_{g\in E} \zeta_h^{\phi'(g)-\phi'(g+e)}
=\sum_{k=0}^{h-1} B_k^{(e)}\zeta_h^k.
\]
If $e\not \in L$ then we use 
Lemma \ref{lemmahprimevanish} and 
$\sum_{k=0}^{h-1} B_k^{(e)}=|E|$ 
to infer
\begin{equation}
    \label{eqacnull}
    AC_{\phi'}(e)=0 
    \  \Leftrightarrow 
    \ B_k^{(e)}= |E|/h\quad \forall \, k.
\end{equation}

Suppose that 
$\phi$ is generalized perfect.
If $e\not \in L$ then 
$|N_{\phi'}^i\cap(e+N_{\phi'}^{i}) |
 =|E|/h^2$
by \eqref{eqacnull} 
and Lemma \ref{ndslemma}.
On the other hand,
$|N_{\phi'}^i\cap(e+N_{\phi'}^{i})|=0$  
if $e\in L\setminus K$,
by Lemma \ref{lem-phi}.
Hence the number of solutions
$(x+K, y+K)\in R\times R$ of
$x+K-(y+K)= e+K$ is $0$ if  
$e\in L\setminus K$
and $|E|/\! \left(|K|h^2\right)$ 
if $e\not \in L$. 
Accordingly $R$ is an
$\left(
\left|E:L\right|,\,
\left|L:K\right|,\,\left|R\right|,
\,|E|/(|K|h^2)\right)$-relative 
difference set in $E/K$, 
with forbidden subgroup $L/K$. 
Also $|E:L| = |G|$,
$|L:K|=h$, and $|R|=|G|$ 
by Lemma~\ref{GammaTransmit}.
Thus $R$ has the claimed parameters.

Now suppose that $R$ is a 
$(v,h,v,v/h)$-relative difference 
set in $E/K$ with forbidden subgroup $L/K$. 
Then  $|N_{\phi'}^i\cap(N_{\phi'}^{i}-e)|
=|E|/h^2$
for any $e \in E \setminus L$; thus 
$B_0^{(e)}=\allowbreak |E|/h$.
Further, if $z\in L_k$ then
$N_{\phi'}^{i-k}-e+z =
N_{\phi'}^i-e$ by Lemma~\ref{lemmaireduce0},
giving $B_0^{(e)}=
B_k^{(e-z)}$.
Since $B_0^{(e)}$ is 
constant as $e$ ranges over 
$E\setminus L$, this means
that
\[
B_0^{(e)}=B_{i}^{(e)}=|E|/h \quad \forall 
\, i \ \, \mbox{and} \ \, 
\forall \, e \not \in L.
\]
By (\ref{eqacnull}), $\phi$ is 
a $\mathrm{GPhA}({\bf s})$. 
\end{proof}

\vspace{1pt}

\begin{proposition}[{\cite[Theorem~4.1]{EFO15}}]
\label{prp-ortho-rds}
Let $H$ be a finite group whose order
is divisible by a prime $h$. 
Then  $\psi \in 
 Z^2(H,\langle \zeta_h\rangle)$ 
is orthogonal if and only if 
$\{(1,x)\hspace{3pt} |\hspace{3pt} x\in H\}$ 
is a $(|H|, h, |H|, 
\allowbreak |H|/h)$-relative 
difference set in $E_\psi$ with 
forbidden subgroup 
$\langle(\zeta_h,1)\rangle$.
\end{proposition}

\begin{theorem}
\label{oc-gpha} 
For prime $h$, a (normalized) $h$-ary 
${\bf s}$-array  $\phi$ is a $GPhA({\bf s})$ 
of type ${\bf z}\not={\bf 0}$ if and only if 
$\mu_{\bf z}\partial\phi$ is orthogonal.
\end{theorem}
\begin{proof}
This is a consequence of 
Theorem~\ref{th-gpha-rds}, 
Proposition~\ref{prp-ortho-rds},
and Lemma~\ref{GammaTransmit}.
\end{proof} 

\vspace{12pt}

The next theorem 
connects generalized plateaued 
functions  to  GPhAs.

\vspace{5pt}

\begin{theorem}
\label{Theo-gpha-pf}
 Let $\phi\colon  {\mathbb{Z}}_{q}^m 
 \rightarrow {\mathbb{Z}}_{h}$ be a map,
 where $h$ is a prime dividing $q$. 
 The following are equivalent$:$
\begin{enumerate}
\item[{\em (1)}] 
$\phi$ is a $GPhA(q^m)$ of 
type $\mathbf{1};$
\item[{\em (2)}] The expansion 
$\phi'\colon {\mathbb{Z}}_{hq}^m 
\rightarrow {\mathbb{Z}}_{h}$ of $\phi$
of type $\mathbf{1}$
is a generalized plateaued function, 
i.e.,
\[
\Big|
    \sum_{x\in{\mathbb{Z}}_{hq}^m} \zeta_h^{\phi'(x)} \zeta_{hq}^{-v\cdot x}
    \Big|^2=\left\{\begin{array}{cl}
       (h^2q)^{m} &  \quad  v\in {\cal F}\\
         0 & \quad 
         v\in \Z_{hq}^m\setminus \mathcal{F},
    \end{array}\right.
    \]
where  ${\cal F}=\{v \in{\mathbb{Z}}_{hq}^m 
\mid  v\equiv \mathbf{1} \, \bmod h\}$.
\end{enumerate}
\end{theorem}
\begin{proof}
 Let $u=(y_1 q, \ldots, y_m q ) \in L$ 
and $v=(y_1' q+a_1, \ldots, \allowbreak 
y_m' q+a_m) \in E= \Z_{hq}^m$ where 
$0\leq y_j, y_j'\leq \allowbreak h-1$
and $0\leq a_j\leq q-1$. Then 
$u \cdot v \equiv 
(a_1y_1+\cdots+a_my_m) q \, \bmod hq$.
Hence, if $\phi$ is a GPhA$(q^m)$ of 
type $\mathbf{1}$,
then by Lemma~\ref{corredftWT}  and
Corollary~\ref{lemmaach},
\[
\Big|\sum_{x\in{\mathbb{Z}}_{hq}^m} 
\zeta_h^{\phi'(x)} \zeta_{hq}^{-v \cdot x}
\Big|^2=\sum_{u\in L} 
AC_{\phi'}(u) 
\zeta_{hq}^{u \cdot v}=
(hq)^m\sum_{0\leq y_1,\ldots,y_m\leq h-1} 
\zeta_{hq}^{-(y_1\, + \, \cdots \, +\, y_m)q 
\, +\, u \cdot v}
\]
The rightmost displayed summation is equal to
\[
(hq)^m\sum_{0\leq y_1,\dots,y_m\leq h-1} 
\zeta_{h}^{(a_1-1)y_1\, +\, \cdots\, 
+\, (a_m-1)y_m}=
\left\{\begin{array}{cl}
 (h^2q)^m    & \quad a_k \equiv 1 
  \, \bmod h\ \ \forall \,k \\
    0 & \quad \mbox{otherwise} .
\end{array}\right.
\]
This proves (1) $\Rightarrow$ (2). 
We get $(2)\Rightarrow (1)$ similarly, 
appealing once more 
to Lemma \ref{corredftWT} and taking
into account that $D_m D_m^*=(hq)^m I_m$.
\end{proof}

\vspace{12pt}

Now we can fulfil our intention
stated just after 
Theorem~\ref{prop-ben-perfe-rds}.

\vspace{5pt}

\begin{theorem}\label{mainaef}
Let $h$ be a prime divisor of 
 $q$, and let 
 $\phi\colon {\mathbb{Z}}_{q}^m 
 \rightarrow \allowbreak {\mathbb{Z}}_{h}$ 
 be an array with expansion
 $\phi'$ 
 of type $\mathbf{z}\neq {\bf 0}$. 
\begin{itemize}
\item[{\rm (a)}]
The following  are equivalent$:$
\begin{enumerate}
 \item[{\rm (i)}] 
 $\mu_{\bf z}\partial\phi$ is 
symmetric and orthogonal, i.e.,  
 $M_{\mu_{\bf z}\partial\phi}$
 is a symmetric Butson Hadamard matrix$;$
\item[{\rm (ii)}] 
$\phi$ is a $GPhA(q^m)$ 
of type  $\mathbf{z};$
\item[{\rm (iii)}] 
$\{g+K\in E/K \mid  \phi'(g)=0\}$ is a 
 non-splitting 
$(q^m,h,q^m,q^{m}/h)$-relative 
difference set in $E/K$
with forbidden subgroup $L/K$.
\end{enumerate}
\item[{\rm (b)}]
If $\mathbf{z}={\bf 1}$
then {\rm (i)}--{\em (iii)}
are equivalent to
\begin{enumerate}
\item[{\rm (iv)}]  $\phi'$ is a 
generalized plateaued function, i.e.,
\[
\Big|
\sum_{x\in{\mathbb{Z}}_{hq}^m} 
\zeta_h^{\phi'(x)} 
\zeta_{hq}^{-v\cdot x}\Big|^2=
\left\{
\begin{array}{cl}
 (h^2q)^{m} &  \quad v\in {\cal F}\\
    0 & \quad \mbox{otherwise}, 
\end{array}\right.
\]
where  ${\cal F}= 
\{ v\in \mathbb{Z}_{hq}^m 
\mid 
v \equiv \mathbf{1} \, \bmod h \}$.
\end{enumerate}
\item[{\rm (c)}] 
Let $h=q$ and 
$\mathbf{z}=\mathbf{1}$.
Suppose that, for all 
$y\in {\mathbb{Z}}_{h}^m\setminus\{\bf 0\}$ with 
$\sum y_i\equiv 0\, \bmod h$, there exists 
$x\in {\mathbb{Z}}_{h}^m$ satisfying
\eqref{eq-phi-lineal}.
Then {\rm (i)}--{\rm (iv)} 
are equivalent to
\begin{enumerate}
\item[{\rm (v)}] $\phi'$ is a GPBF.
\end{enumerate}
\end{itemize}
\end{theorem}
\begin{proof}
The equivalences ${\rm (i)}\Leftrightarrow{\rm (ii)}$, 
${\rm (ii)}\Leftrightarrow{\rm (iii)}$, 
${\rm (ii)}\Leftrightarrow {\rm (iv)}$, and 
${\rm (ii)}\Leftrightarrow{\rm (v)}$ follow from, 
respectively,  Theorems~\ref{oc-gpha},  
\ref{th-gpha-rds}, \ref{Theo-gpha-pf}, and 
Proposition~\ref{pro-gpbf-gpha}.
(Proposition~\ref{pro-fz}~(ii) justifies
 non-splitting in (iii).)
\end{proof}

\begin{remark}
\label{GPBFCoinc}
{\em 
In \cite[Definition 2.2]{WZ07},
 a map $\phi\colon
 {\mathbb{Z}}_{q}^m 
 \rightarrow {\mathbb{Z}}_{q}$ 
 is called a  generalized 
 partially bent function if
$(q^m-N_F)(q^m-N_C)=q^m$ where $N_F=
|\{v\in 
{\mathbb{Z}}_{q}^m \mid
\sum_{x\in{\mathbb{Z}}_{q}^m} 
\zeta_q^{\phi(x)-v\cdot x}=0 \}|$ 
and $N_C=|\{v\in 
{\mathbb{Z}}_{q}^m \mid 
AC_{-\phi}(v)=0 \}|$.
Previously, Carlet~\cite[Definition 1]{Car93}
introduced partially bent functions
for $q=2$. The coincidence with 
 our Definition~\ref{dfgpbf} is 
 shown in \cite[Theorem~2]{Wang97} 
and  \cite[Proposition 8]{MOS18}. 
Observe that, for
$\mathbf{z}=\mathbf{1}$, we have
$|L| \cdot |{\cal F}|= 
\allowbreak (hq)^m$,
$|L|=(hq)^m-N_C$, and 
$|{\cal F}|=(hq)^m-N_F$.}
\end{remark}

\begin{remark}
\label{Remark4dot14}
{\em For  $q$ prime, if
$\phi:\Z_q^m\rightarrow \Z_q$ is 
a $\mathrm{GPqA}(q^m)$
of type $\bf 1$, 
then $\phi'$ 
is a $2m$-generalized 
plateaued function.}
\end{remark}

\section{Examples}\label{sec:examples}

\begin{example}\label{egbpa}
{\em 
Let $\phi$ be the map on 
${\mathbb Z}_2^3$  with layers
\[
A_0=
{\scriptsize
\begin{bmatrix}
0 &    \ 1 \\
1 &  \ 1 
\end{bmatrix}}
\quad \mbox{and} \quad  
A_1={\scriptsize
\begin{bmatrix}
0 & \ 1
 \\
0 & \ 0  
\end{bmatrix}}.
\]
Here $A_i$ 
is the layer on 
$\{i\} \times {\mathbb Z}_2 
 \times  {\mathbb Z}_2$, and 
 $\phi(i,j,k)=A_i(j,k)$. Then 
$\phi$ is a GPBA($2^3$), i.e.,
a GP$2$A($2^3$), 
of type $\bf 1$.
It has orthogonal cocycle  
$\mu_{\bf 1}\partial_2\partial_3
\partial_4\partial_6$, where  
$\partial_i=\partial \phi_i$
for the multiplicative 
Kronecker delta $\phi_i$ of $\alpha_i$, 
with $\alpha_0=(0,0,0)$,
 $\alpha_1=(0,0,1)$, 
and so on. 
%
We label rows and columns with the 
elements of 
$\Z_2^3=\{\alpha_0,\ldots,\alpha_7\}$
in this ordering, and display the 
cocyclic Hadamard matrix
$M_{\mu_{\bf 1}\partial\phi}$
as a Hadamard (entrywise) product 
$M_{\mu_{\bf 1}}\circ M_{\partial\phi}$
in logarithmic form:
\[
{\scriptsize
\begin{bmatrix}
0 & 0 & 0 & 0 & 0 & 0 & 0 & 0  \\
0 & 1 & 0 & 1 & 0 & 1 & 0 & 1  \\
0 & 0 & 1 & 1 & 0 & 0 & 1 & 1  \\
0 & 1 & 1 & 0 & 0 & 1 & 1 & 0  \\
0 & 0 & 0 & 0 & 1 & 1 & 1 & 1  \\
0 & 1 & 0 & 1 & 1 & 0 & 1 & 0  \\
0 & 0 & 1 & 1 & 1 & 1 & 0 & 0  \\
0 & 1 & 1 & 0 & 1 & 0 & 0 & 1 
\end{bmatrix}
}
\circ
{\scriptsize
\begin{bmatrix}
0 & 0 & 0 & 0 & 0 & 0 & 0 & 0  \\
0 & 0 & 1 & 1 & 0 & 0 & 1 & 1  \\
0 & 1 & 0 & 1 & 1 & 0 & 1 & 0  \\
0 & 1 & 1 & 0 & 1 & 0 & 0 & 1  \\
0 & 0 & 1 & 1 & 0 & 0 & 1 & 1  \\
0 & 0 & 0 & 0 & 0 & 0 & 0 & 0  \\
0 & 1 & 1 & 0 & 1 & 0 & 0 & 1   \\
0 & 1 & 0 & 1 & 1 & 0 & 1 & 0   
\end{bmatrix}
}
=
{\scriptsize
\begin{bmatrix}
0 & 0 & 0 & 0 & 0 & 0 & 0 & 0  \\
0 & 1 & 1 & 0 & 0 & 1 & 1 & 0  \\
0 & 1 & 1 & 0 & 1 & 0 & 0 & 1  \\
0 & 0 & 0 & 0 & 1 & 1 & 1 & 1  \\
0 & 0 & 1 & 1 & 1 & 1 & 0 & 0  \\
0 & 1 & 0 & 1 & 1 & 0 & 1 & 0  \\
0 & 1 & 0 & 1 & 0 & 1 & 0 & 1  \\
0 & 0 & 1 & 1 & 0 & 0 & 1 & 1  
\end{bmatrix}.
}
\]
The expansion 
$\phi' \colon {\mathbb{Z}}_{4}^3 
\rightarrow {\mathbb{Z}}_{2}$ 
is defined by 
the layers $B_i$ on 
$\{i\}\times \Z_4\times\Z_4$,  
$0\leq i\leq 3$, where
\[
B_i=\left\{
\begin{array}{cc}
  {\small 
\left[\begin{array}{ccrrr}
A_0 &  A_0 \oplus J  \\ 
A_0 \oplus J  & A_0 
\end{array}
\right]}   &  i=0,2 \\[7mm]
{\small  
\left[\begin{array}{ccrrr}
A_1 \oplus J  &  
A_1 \\
A_1 &  A_1 \oplus J  
\end{array}
\right]}
 & i=1,3, 
\end{array}
\right.
\]
 $J$ denoting the all $1$s matrix.
 We have
$L=\{(0,0,0),(0,0,2),(0,2,0),(0,2,2),
(2,0,0),(2,0,2),
\allowbreak (2,2,0),(2,2,2)\}$, 
\[
AC_{\phi'}(v)=\left\{\begin{array}{cl}
   (-1)^{\mathrm{wt}(v)} \, 64 &  \ v\in L\\
0     & \ v\notin L,
\end{array}\right.
\]
${\cal F}=\{(1,1,1),(1,1,3),(1,3,1),
(1,3,3), (3,1,1),(3,1,3),(3,3,1),
(3,3,3)\}$, and 
\[
 \Big|\sum_{x\in{\mathbb{Z}}_{4}^3} 
 \zeta_2^{\phi'(x)} 
 \zeta_4^{-v\cdot x}\Big|^2=
 \left\{\begin{array}{cl}
    512 &  \ v\in {\cal F}\\
0     & \ v\notin{\cal F}.
\end{array}\right.
\]
Therefore $\phi'$ is a GPBF.} 
\end{example}

\begin{example}\label{egsixteen}
{\em The map $\phi={\scriptsize
\left[
\begin{array}{cccc}
0     &  1  & 1 & 1 \\
1     &  1  & 0 & 1 \\
0      & 1 & 0  & 0 \\
0 & 0 & 0 & 1
\end{array}
\right]
}$
on $\Z_4^2$ is a GPBA$(4^2)$ 
of type $\bf 1$.
Its orthogonal cocycle is  
$\mu_{\bf 1}\partial\phi$.
If we label rows and columns
with the elements of 
$\Z_4^2=\{\alpha_0=(0,0),\alpha_1=(0,1),
\alpha_2=(0,2),\ldots,\alpha_{15}=(3,3)\}$, 
then the cocyclic Hadamard matrix
$M_{\mu_{\bf 1}}\circ M_{\partial \phi}$
 in logarithmic form is} 
\[
{\footnotesize
\renewcommand{\arraycolsep}{.15cm}\left[
\begin{array}{cccccccccccccccc}
 0 &  0 &  0 &  0 &  0 &  0 &  0 &  0 &  0 &  0 &  0 &  0 &  0 &  0 &  0 &  0 \\
 0 &  0 &  0 &  1 &  0 &  0 &  0 &  1 &  0 &  0 &  0 &  1 &  0 &  0 &  0 &  1 \\
 0 &  0 &  1 &  1 &  0 &  0 &  1 &  1 &  0 &  0 &  1 &  1 &  0 &  0 &  1 &  1 \\
 0 &  1 &  1 &  1 &  0 &  1 &  1 &  1 &  0 &  1 &  1 &  1 &  0 &  1 &  1 &  1 \\
 0 &  0 &  0 &  0 &  0 &  0 &  0 &  0 &  0 &  0 &  0 &  0 &  1 &  1 &  1 &  1 \\
 0 &  0 &  0 &  1 &  0 &  0 &  0 &  1 &  0 &  0 &  0 &  1 &  1 &  1 &  1 &  0 \\
 0 &  0 &  1 &  1 &  0 &  0 &  1 &  1 &  0 &  0 &  1 &  1 &  1 &  1 &  0 &  0 \\
 0 &  1 &  1 &  1 &  0 &  1 &  1 &  1 &  0 &  1 &  1 &  1 &  1 &  0 &  0 &  0 \\
 0 &  0 &  0 &  0 &  0 &  0 &  0 &  0 &  1 &  1 &  1 &  1 &  1 &  1 &  1 &  1 \\
 0 &  0 &  0 &  1 &  0 &  0 &  0 &  1 &  1 &  1 &  1 &  0 &  1 &  1 &  1 &  0 \\
 0 &  0 &  1 &  1 &  0 &  0 &  1 &  1 &  1 &  1 &  0 &  0 &  1 &  1 &  0 &  0 \\
 0 &  1 &  1 &  1 &  0 &  1 &  1 &  1 &  1 &  0 &  0 &  0 &  1 &  0 &  0 &  0 \\
 0 &  0 &  0 &  0 &  1 &  1 &  1 &  1 &  1 &  1 &  1 &  1 &  1 &  1 &  1 &  1 \\
 0 &  0 &  0 &  1 &  1 &  1 &  1 &  0 &  1 &  1 &  1 &  0 &  1 &  1 &  1 &  0 \\
 0 &  0 &  1 &  1 &  1 &  1 &  0 &  0 &  1 &  1 &  0 &  0 &  1 &  1 &  0 &  0 \\
 0 &  1 &  1 &  1 &  1 &  0 &  0 &  0 &  1 &  0 &  0 &  0 &  1 &  0 &  0 &  0 
\end{array}\right]
}
\circ
{\footnotesize
\renewcommand{\arraycolsep}{.15cm}\left[
\begin{array}{cccccccccccccccc}
 0 &  0 &  0 &  0 &  0 &  0 &  0 &  0 &  0 &  0 &  0 &  0 &  0 &  0 &  0 &  0  \\
  0 &  0 &  1 &  1 &  1 &  1 &  0 &  0 &  0 &  1 &  1 &  0 &  1 &  0 &  0 &  1  \\
  0 &  1 &  0 &  1 &  0 &  1 &  0 &  1 &  1 &  0 &  1 &  0 &  1 &  0 &  1 &  0  \\
  0 &  1 &  1 &  0 &  1 &  0 &  0 &  1 &  1 &  1 &  0 &  0 &  0 &  0 &  1 &  1  \\
  0 &  1 &  0 &  1 &  1 &  0 &  0 &  1 &  1 &  0 &  1 &  0 &  0 &  1 &  1 &  0  \\
  0 &  1 &  1 &  0 &  0 &  0 &  0 &  0 &  1 &  1 &  0 &  0 &  1 &  0 &  1 &  0  \\
  0 &  0 &  0 &  0 &  0 &  0 &  1 &  1 &  0 &  0 &  0 &  0 &  0 &  0 &  1 &  1  \\
  0 &  0 &  1 &  1 &  1 &  0 &  1 &  0 &  0 &  1 &  1 &  0 &  1 &  1 &  1 &  1  \\
  0 &  0 &  1 &  1 &  1 &  1 &  0 &  0 &  0 &  0 &  1 &  1 &  1 &  1 &  0 &  0  \\
  0 &  1 &  0 &  1 &  0 &  1 &  0 &  1 &  0 &  0 &  0 &  0 &  0 &  0 &  0 &  0  \\
  0 &  1 &  1 &  0 &  1 &  0 &  0 &  1 &  1 &  0 &  0 &  1 &  0 &  1 &  1 &  0  \\
  0 &  0 &  0 &  0 &  0 &  0 &  0 &  0 &  1 &  0 &  1 &  0 &  1 &  0 &  1 &  0  \\
  0 &  1 &  1 &  0 &  0 &  1 &  0 &  1 &  1 &  0 &  0 &  1 &  1 &  0 &  1 &  0  \\
  0 &  0 &  0 &  0 &  1 &  0 &  0 &  1 &  1 &  0 &  1 &  0 &  0 &  0 &  1 &  1  \\
  0 &  0 &  1 &  1 &  1 &  1 &  1 &  1 &  0 &  0 &  1 &  1 &  1 &  1 &  1 &  1  \\
  0 &  1 &  0 &  1 &  0 &  0 &  1 &  1 &  0 &  0 &  0 &  0 &  0 &  1 &  1 &  0 \end{array}\right]
}
\]
\[
=
{\footnotesize
\renewcommand{\arraycolsep}{.15cm}
\left[
\begin{array}{cccccccccccccccc}
0 & 0 & 0 & 0 & 0 & 0 & 0 & 0 & 0 & 0 & 0 & 0 & 0 & 0 & 0 & 0 \\ 
0 & 0 & 1 & 0 & 1 & 1 & 0 & 1 & 0 & 1 & 1 & 1 & 1 & 0 & 0 & 0 \\ 
0 & 1 & 1 & 0 & 0 & 1 & 1 & 0 & 1 & 0 & 0 & 1 & 1 & 0 & 0 & 1 \\ 
0 & 0 & 0 & 1 & 1 & 1 & 1 & 0 & 1 & 0 & 1 & 1 & 0 & 1 & 0 & 0 \\ 
0 & 1 & 0 & 1 & 1 & 0 & 0 & 1 & 1 & 0 & 1 & 0 & 1 & 0 & 0 & 1 \\ 
0 & 1 & 1 & 1 & 0 & 0 & 0 & 1 & 1 & 1 & 0 & 1 & 0 & 1 & 0 & 0 \\ 
0 & 0 & 1 & 1 & 0 & 0 & 0 & 0 & 0 & 0 & 1 & 1 & 1 & 1 & 1 & 1 \\ 
0 & 1 & 0 & 0 & 1 & 1 & 0 & 1 & 0 & 0 & 0 & 1 & 0 & 1 & 1 & 1 \\ 
0 & 0 & 1 & 1 & 1 & 1 & 0 & 0 & 1 & 1 & 0 & 0 & 0 & 0 & 1 & 1 \\ 
0 & 1 & 0 & 0 & 0 & 1 & 0 & 0 & 1 & 1 & 1 & 0 & 1 & 1 & 1 & 0 \\ 
0 & 1 & 0 & 1 & 1 & 0 & 1 & 0 & 0 & 1 & 0 & 1 & 1 & 0 & 1 & 0 \\ 
0 & 1 & 1 & 1 & 0 & 1 & 1 & 1 & 0 & 0 & 1 & 0 & 0 & 0 & 1 & 0 \\ 
0 & 1 & 1 & 0 & 1 & 0 & 1 & 0 & 0 & 1 & 1 & 0 & 0 & 1 & 0 & 1 \\ 
0 & 0 & 0 & 1 & 0 & 1 & 1 & 1 & 0 & 1 & 0 & 0 & 1 & 1 & 0 & 1 \\ 
0 & 0 & 0 & 0 & 0 & 0 & 1 & 1 & 1 & 1 & 1 & 1 & 0 & 0 & 1 & 1 \\ 
0 & 0 & 1 & 0 & 1 & 0 & 1 & 1 & 1 & 0 & 0 & 0 & 1 & 1 & 1 & 0  
\end{array}\right].
}
\]
{\em 
The expansion $\phi' 
\colon {\mathbb{Z}}_{8}^2 
\rightarrow {\mathbb{Z}}_{2}$ 
is defined by}
\[
{\scriptsize
\begin{bmatrix}
0 & 1 & 1 & 1 & 1 & 0 & 0 & 0 \\
1 & 1 & 0 & 1 & 0 & 0 & 1 & 0 \\
0 & 1 & 0 & 0 & 1 & 0 & 1 & 1 \\
0 & 0 & 0 & 1 & 1 & 1 & 1 & 0 \\
1 & 0 & 0 & 0 & 0 & 1 & 1 & 1 \\
0 & 0 & 1 & 0 & 1 & 1 & 0 & 1 \\
1 & 0 & 1 & 1 & 0 & 1 & 0 & 0 \\
1 & 1 & 1 & 0 & 0 & 0 & 0 & 1
  
\end{bmatrix},
}
\]
{\em 
with $L=\{(0,0),(0,4),(4,0),(4,4)\}$, 
\[
AC_{\phi'}(v)=
\left\{\begin{array}{cl}
  (-1)^{\mathrm{wt}(v)}\, 64 &  \ \ v\in L\\
0     & \ \ v\notin L,
\end{array}\right.
\]
\begin{align*}
{\cal F}&=\{(1,1),(1,3),(1,5),(1,7), 
(3,1),(3,3),(3,5),(3,7), (5,1),(5,3),
(5,5),(5,7),\\
& \hspace{20pt} (7,1),(7,3), (7,5),(7,7) \},
\end{align*} 
and 
\[
 \Big|\sum_{x\in{\mathbb{Z}}_{8}^2} 
 \zeta_2^{\phi'(x)} \zeta_8^{-v\cdot x}
 \Big|^2=
 \left\{
 \begin{array}{cl}
    256 & \ v\in {\cal F}\\
     0  & \ v\notin{\cal F}.
\end{array}\right.
\]
Therefore $\phi'$ is a GPBF. }
\end{example}

\begin{example}\label{egthirdroots}
{\em 
The map 
$\phi = {\small
\left[
\begin{array}{ccc}
0     &  0  & 0 \\
0     &  1  & 0 \\
2      & 2 & 1
\end{array}
\right]
}$
on $\Z_3^2$ is a GP$3$A$(3^2)$ 
of type $\bf 1$. Its 
orthogonal cocycle is  
$\mu_{\bf 1}\partial\phi$.
Labeling the rows and columns with 
the elements of $\Z_3^2=
\{\alpha_0=(0,0),\alpha_1=(0,1),\alpha_2=(0,2),
\ldots,\alpha_8=(2,2)\}$,
we display the cocyclic Butson matrix
$M_{\mu_{\bf 1}}\circ M_{\partial\phi}$:}
\[
{\scriptsize
\begin{bmatrix}
0 & 0 & 0 & 0 & 0 & 0 & 0 & 0 & 0 \\
0 & 0 & 1 & 0 & 0 & 1 & 0 & 0 & 1 \\
0 & 1 & 1 & 0 & 1 & 1 & 0 & 1 & 1 \\
0 & 0 & 0 & 0 & 0 & 0 & 1 & 1 & 1 \\
0 & 0 & 1 & 0 & 0 & 1 & 1 & 1 & 2 \\
0 & 1 & 1 & 0 & 1 & 1 & 1 & 2 & 2 \\
0 & 0 & 0 & 1 & 1 & 1 & 1 & 1 & 1 \\
0 & 0 & 1 & 1 & 1 & 2 & 1 & 1 & 2 \\
0 & 1 & 1 & 1 & 2 & 2 & 1 & 2 & 2 
\end{bmatrix}
}
\circ
{\scriptsize
\begin{bmatrix}
0 & 0 & 0 & 0 & 0 & 0 & 0 & 0 & 0 \\
0 & 0 & 0 & 1 & 2 & 0 & 0 & 2 & 1 \\
0 & 0 & 0 & 0 & 2 & 1 & 2 & 0 & 1 \\
0 & 1 & 0 & 2 & 1 & 1 & 1 & 1 & 2 \\
0 & 2 & 2 & 1 & 2 & 1 & 0 & 0 & 1 \\
0 & 0 & 1 & 1 & 1 & 2 & 1 & 1 & 2 \\
0 & 0 & 2 & 1 & 0 & 1 & 2 & 0 & 0 \\
0 & 2 & 0 & 1 & 2 & 1 & 0 & 2 & 0 \\
0 & 1 & 1 & 2 & 1 & 2 & 0 & 0 & 2 
\end{bmatrix}
}
=
{\scriptsize
\begin{bmatrix}
0 & 0 & 0 & 0 & 0 & 0 & 0 & 0 & 0 \\
0 & 0 & 1 & 1 & 2 & 1 & 0 & 2 & 2 \\
0 & 1 & 1 & 0 & 0 & 2 & 2 & 1 & 2 \\
0 & 1 & 0 & 2 & 1 & 1 & 2 & 2 & 0 \\
0 & 2 & 0 & 1 & 2 & 2 & 1 & 1 & 0 \\
0 & 1 & 2 & 1 & 2 & 0 & 2 & 0 & 1 \\
0 & 0 & 2 & 2 & 1 & 2 & 0 & 1 & 1 \\
0 & 2 & 1 & 2 & 1 & 0 & 1 & 0 & 2 \\
0 & 2 & 2 & 0 & 0 & 1 & 1 & 2 & 1 
\end{bmatrix}.
}
\]
\noindent
{\em 
The expansion $\phi' 
\colon {\mathbb{Z}}_{9}^2 
\rightarrow {\mathbb{Z}}_{3}$ 
is defined by}
\[
{\scriptsize
\begin{bmatrix}
    0 & 0 & 0 & 1 & 1 & 1 & 2 & 2 & 2 \\
 0 & 1 & 0 & 1 & 2 & 1 & 2 & 0 & 2 \\
 2 & 2 & 1 & 0 & 0 & 2 & 1 & 1 & 0 \\
 1 & 1 & 1 & 2 & 2 & 2 & 0 & 0 & 0 \\
 1 & 2 & 1 & 2 & 0 & 2 & 0 & 1 & 0 \\
 0 & 0 & 2 & 1 & 1 & 0 & 2 & 2 & 1 \\
 2 & 2 & 2 & 0 & 0 & 0 & 1 & 1 & 1 \\
 2 & 0 & 2 & 0 & 1 & 0 & 1 & 2 & 1 \\
 1 & 1 & 0 & 2 & 2 & 1 & 0 & 0 & 2
\end{bmatrix},
}
\]
\noindent
{\em 
with $L=\{(0,0),(0,3),(0,6),(3,0),
(3,3),(3,6),(6,0),(6,3),(6,6)\}$,
\[
AC_{\phi'}((v_1,v_2))=
\left\{\begin{array}{cl}
  81 \,\zeta_3^{-(v_1+v_2)/3}  
  &\  \ (v_1,v_2)\in L\\
  0 & \ \ (v_1,v_2)\notin L,
\end{array}\right.
\]
${\cal F}=\{(1,1),(1,4),(1,7),(4,1), (4,4),(4,7), 
(7,1),(7,4),(7,7)\}$, and 
\[
 \Big|\sum_{x\in{\mathbb{Z}}_{9}^2} 
  \zeta_3^{\phi'(x)} \zeta_9^{-v\cdot x}
 \Big|^2=
 \left\{\begin{array}{cl}
    729 & \ \ v\in {\cal F}\\
     0  & \ \ v\notin{\cal F}.
\end{array}\right.
\]
Therefore $\phi'$ is a
GPBF.
Also $\phi'$ is a  
$4$-generalized plateaued function
(see Remark~\ref{Remark4dot14}).}
\end{example}

It may be checked that the sufficient 
condition \eqref{eq-phi-lineal} is satisfied 
in each of the
Examples \ref{egbpa}, \ref{egsixteen}, 
\ref{egthirdroots}.

We now recite a bit more 
algebraic design theory
 in preparation for 
our closing result, which provides
an infinite family of 
GPhAs of type $\bf 1$.

\begin{proposition}[cf.~{\cite[Theorem 15.8.4]{ADT}}]
\label{KroneckerBuild}
Let $G_{\bf s} = \mathbb{Z}_{s_1} \times \cdots 
\times \mathbb{Z}_{s_{m}}$, 
$G_{\bf t} = \mathbb{Z}_{t_1} \times \cdots 
\times \mathbb{Z}_{t_{n}}$, and 
$G = G_{\bf s} \times G_{\bf t}$. 
Suppose that 
$\psi\partial\phi_{\bf s} 
\in  Z^2( G_{\bf s}
, \langle \zeta_{k_{1}} \rangle)$ and 
$\rho\partial\phi_{\bf t}\in 
Z^2( G_{\bf t},
\langle \zeta_{k_{2}} \rangle)$ 
are orthogonal. 
Let $k = \mathrm{lcm}(k_{1},k_{2})$.
Define $\varphi\in
Z^2(G, \langle \zeta_{k} \rangle)$ by
$\varphi(g_{s}g_{t},h_{s}h_{t}) = 
\psi(g_{s},h_{s})\rho(g_{t},h_{t})$,  
and define a map $\phi$ on $G$
by $\phi(g_{s}g_{t}) = 
\phi_{\bf s}(g_{s}) \phi_{\bf t}(g_{t})$.
Then $\varphi\partial\phi\in
Z^2(G,\langle \zeta_{k} \rangle)$
is orthogonal, with cocyclic matrix 
$[\psi\partial\phi_{\bf s}]\otimes 
[\rho\partial\phi_{\bf t}]$.
\end{proposition}

\begin{corollary}
Let $h = q$ be prime. 
If there exist symmetric cocyclic 
matrices in $\mathrm{BH}(q^m,h)$ 
and $\mathrm{BH}(q^n,h)$, 
corresponding to a $GPhA(q^m)$ of 
type ${\bf 1}$ and 
a $GPhA(q^n)$ of type ${\bf 1}$,
respectively, then there 
exists a symmetric cocyclic matrix 
in $\mathrm{BH}(q^{m+n},h)$ 
corresponding to 
a $GPhA(q^{m+n})$ of type ${\bf 1}$.
\end{corollary}

Example \ref{egbpa} furnishes a 
symmetric orthogonal 
cocycle $\mu_{\bf 1}\partial\phi\in
Z^{2}(\Z_{2}^{3},\Z_{2})$ with nontrivial 
coboundary $\partial \phi$. 
By iteration of
Proposition \ref{KroneckerBuild}, 
Kronecker multiplying
$\mu_{\bf 1}\partial\phi$ by powers of 
$\mu_{\bf 1}\in Z^2(\Z_{2},\Z_{2})$, 
we get a symmetric orthogonal cocycle 
$\mu_{\bf 1}\partial\chi\in
Z^{2}(\Z_{2}^{k},\Z_{2})$ 
for all $k \geq 3$. 
Then $\chi$ is a 
$\mathrm{GPBA}(2^k)$ 
of type ${\bf 1}$, 
and Theorem \ref{mainaef} constructs 
its associated objects. 
Thus, for all $k\geq 3$ there exists 
a map $\Z_2^k\to \Z_2$ whose 
expansion is a GPBF; whereas for 
odd $k$, we recall that no GBF---i.e., 
no bent function---can exist.


\bigskip

\begin{thebibliography}{99}

\bibitem{GOBA}
J.~A.~Armario and D.~L.~Flannery,
 {\em Generalized binary arrays from 
quasi-orthogonal cocycles}, 
Des. Codes Cryptogr. 87 (2019), no.~10,
2405--2417.

\bibitem{Car93} C. Carlet, 
{\em Partially-bent functions}. 
Des. Codes Cryptogr. 3 (1993), no.~2,
135--145.

\bibitem{ADT} 
W. de Launey and D.~L.~Flannery,
{\em Algebraic design theory}. 
Math. Surveys. Monogr. 175, 
American Mathematical Society, 
Providence, RI, 2011.

\bibitem{EFO15} 
R. Egan, D.~L.~Flannery, and 
P. \'{O} Cath\'{a}in, 
{\em Classifying cocyclic Butson 
Hadamard matrices}. 
In: Colbourn, C. (Ed.) 
\textit{Algebraic Design Theory 
and Hadamard Matrices}, 
Springer Proc. Math. Stat, 
vol. 133, pp. 93--106, 2015.

\bibitem{ParsevalThm}
D.~F.~Elliott and K.~R.~Rao,
{\em Fast Transforms: Algorithms, 
 Analyses, Applications},
Academic Press, Inc., USA, 1982.

\bibitem{Horadam}
K.~J. Horadam, {\em Hadamard matrices 
and their applications}, 
Princeton University Press, Princeton, NJ, 2007.

\bibitem{Hug00} G. Hughes, 
{\em Non-splitting abelian 
$(4t,2,4t, 2t)$ relative 
difference sets and Hadamard cocycles}. 
Europ. J. Combin.~21 (2000), 
no.~3, 323--331.

\bibitem{Jed92} J. Jedwab, 
{\em Generalized perfect arrays and Menon 
difference sets}.
Des. Codes Cryptogr.~2 (1992), 
no.~1, 19--68.

\bibitem{KSW85}
P.~V.~Kumar, R.~A.~Scholtz, and L.~R.~Welch,
{\em Generalised bent functions and 
 their properties}.
J. Combin. Theory Ser. A, 40, (1985), 90--107.

\bibitem{LeungSchmidt19}
K. H. Leung and B. Schmidt, 
{\em Nonexistence results on 
generalized bent functions 
{$\mathbb{Z}_q^m\to\mathbb{Z}_q$} 
with odd {$m$} and {$q\equiv 2 \pmod 4$}}, 
J. Combin. Theory Ser. A, 163, (2019), 
1--33.

\bibitem{MW-S}
F. J. MacWilliams and N. J. A. Sloane,
{\em The theory of error-correcting codes. {II}}.
North-Holland Mathematical Library, vol. 16, 
North-Holland Publishing Co., 
Amsterdam-New York-Oxford, 1977.

\bibitem{MOS18}
S. Mesnager, F. \"{O}zbudak, 
and A. S\i nak, 
{\em Characterizations of 
partially bent and plateaued 
functions over finite fields}. 
Arithmetic of Finite Fields,
224--241, Lecture Notes in Comput. Sci., 
vol.~11321, 
Springer, Cham, 2018.

\bibitem{MTQ17} 
S. Mesnager, C. Tang, and Y. Qi, 
{\em Generalized plateaued functions 
and admissible (plateaued) functions}.
IEEE Trans. Inf. Theory 63 (2017), no.~10, 
6139--6148.

\bibitem{Sch19}
B. Schmidt, {\em A survey of 
group invariant Butson 
matrices and their relation 
to generalized bent functions 
and various other objects.} 
Radon Ser. Comput. Appl. Math. 23 
(2019), 241--251.

\bibitem{Wang97} J. Wang, 
{\em The linear kernel of Boolean 
functions and partially bent functions}. 
Systems Sci. Math. Sci. 10 (1997), no.~1,
6--11.

\bibitem{WZ07}
X. Wang and J. Zhou, 
{\em Generalized partially bent functions}.
In: Future Generation Communication and 
Networking (FGCN 2007). vol.~1, pp. 16--21, 
IEEE, 2007.
\end{thebibliography}
\end{document}